\documentclass[11pt,reqno]{amsart}
\usepackage[a4paper, tmargin=3cm, bmargin=3.0cm, lmargin=2.0cm, rmargin=2.0cm, textheight=24cm, textwidth=16cm]{geometry}

\usepackage[breaklinks=true]{hyperref}

\usepackage{setspace}
\usepackage{caption,color}
\usepackage{amsmath,amsthm,amsfonts,amssymb}
\usepackage{fullpage}
\theoremstyle{plain}
\newtheorem{theorem}{Theorem}[section]

\newtheorem{lemma}[theorem]{Lemma}

\newtheorem{rem}[theorem]{Remark}

\DeclareMathOperator{\inte}{int}
\DeclareMathOperator{\diag}{diag}

\theoremstyle{definition}
\newtheorem{defn}[theorem]{Definition}
\begin{document}
\title{Interval hulls of $N$-matrices and almost $P$-matrices}

\author{Projesh Nath Choudhury}
\address[Projesh Nath ~Choudhury]{Department of Mathematics, Indian Institute of
	Science, Bangalore 560012, India}
\email{\tt projeshc@iisc.ac.in, projeshnc@alumni.iitm.ac.in}

\author{M. Rajesh Kannan}
\address[M. Rajesh~Kannan]{Department of Mathematics, Indian Institute of
	Technology Kharagpur, Kharagpur 721302, India}
\email{\tt rajeshkannan@maths.iitkgp.ac.in, rajeshkannan1.m@gmail.com}
\date{\today}
    \begin{abstract}
    	We establish a characterization of almost $P$-matrices via a sign non-reversal property. In this we are inspired by the analogous results for $N$-matrices. Next, the interval hull of two $m \times n$ matrices $A=(a_{ij})$ and $B = (b_{ij})$, denoted by $\mathbb{I}(A,B)$, is the collection of all matrices $C \in \mathbb{R}^{m \times n}$ such that each $c_{ij}$ is a convex combination of $a_{ij}$ and $b_{ij}$. Using the sign non-reversal property, we identify a finite subset of  $\mathbb{I}(A,B)$ that determines if all matrices in $\mathbb{I}(A,B)$ are $N$-matrices/almost $P$-matrices. This provides a test for an entire class of matrices simultaneously to be $N$-matrices/almost $P$-matrices. We also establish analogous results for semipositive and minimally semipositive matrices. These characterizations may be considered similar in
    spirit to that of $P$-matrices by Bia{\l}as--Garloff [\textit{Linear Algebra Appl.} 1984] and Rohn--Rex [\textit{SIMAX} 1996], and of positive definite matrices by Rohn [\textit{SIMAX} 1994].
    \end{abstract}
\subjclass[2010]{15B48 (primary), 15A24, 65G30 (secondary)}
\keywords{Sign non-reversal property, interval hull of matrices, $N$-matrices, almost $P$-matrices, semipositive matrices}
\maketitle
\section{Introduction}

Let $\mathbb{R}^{m \times n }$ denote the space of all real $m \times n$ matrices. For $A=(a_{ij}), B=(b_{ij}) \in \mathbb{R}^{m \times n }$, the interval hull of the matrices $A$ and $B$, denoted by  $\mathbb{I}(A,B)$, is
defined as follows:
\begin{center}
	$\mathbb{I}(A,B) = \{C \in \mathbb{R}^{m \times n}: c_{ij} = t_{ij} a_{ij} + (1 - t_{ij}) b_{ij}, t_{ij} \in [0,1]\}$.
\end{center}

If $A\neq B$, the interval hull
contains uncountably many matrices. One of the interesting questions,
related to interval hulls of matrices, considered in the literature is
the following:  Suppose a finite subset of matrices in $\mathbb{I}(A, B)$
has some property, say $\mathcal{S}$. Does the entire interval hull
$\mathbb{I}(A, B)$ have the property $\mathcal{S}$?  For example, it was shown in \cite{Ku71} that if the
matrices $A$ and $B$ are invertible and $A \leq B$ (entry wise), then all
the matrices in the interval hull $\mathbb{I}(A, B)$ are invertible if
and only if ($A,B$ are invertible and) $A^{-1}, B^{-1}$ are entrywise nonnegative. In \cite{rohn1}, the author considered the positive definiteness and stability of the interval hulls of matrices. For a collection of several matrix classes having interval hull characterizations, see the recent survey \cite{GAT16}. An interval hull $\mathbb{I}(A, B)$ is
\textit{of type} $\mathcal{S}$ if all the matrices in $\mathbb{I}(A,
B)$ are of type $\mathcal{S}$.

In this article, we provide necessary and sufficient conditions for an interval hull of matrices $\mathbb{I}(A,B)$ to be contained in one of the following classes: $N$-matrices, almost $P$-matrices, (minimally) semipositive matrices, by reducing it in each case to a finite set of test matrices.

\begin{defn}\label{maindefn}
	Let $m,n \geq 1$ be integers.
		\begin{enumerate}
		\item A matrix $A \in \mathbb{R}^{n \times n}$ is   \emph{an $N$-matrix} if all its principal minors are negative.
		\item A matrix $A \in \mathbb{R}^{n \times n} ~(n \geq 2)$ is an \emph{almost $P$-matrix} if all its proper principal minors are positive, and the determinant of $A$ is negative.
		\item  An $m \times n$ real matrix $A$ is a  \emph{semipositive matrix},  if there exists a vector $x \geq 0$ such that $Ax > 0$. An $m \times n$ real matrix $A$ is a  \emph{minimally semipositive matrix}, if it is semipositive and no column-deleted submatrix of $A$ is semipositive.
	\end{enumerate}
\end{defn}

$N$-matrices was introduced by Inada in 1971 \cite{I71}. These matrices have rich applications in univalence theory (injectivity of
differential maps in $\mathbb{R}^n$) and the Linear Complementarity
Problem \cite{par-rav-nmat}. Recently in \cite{PT20}, the first author in
joint work with Tsatsomeros has established an algorithm to detect
whether a given matrix is an $N$-matrix or not; as well as an algorithm
to construct every $N$-matrix recursively. The sign non-reversal property
for $N$-matrices was established by Mohan--Sridhar \cite{moh-sri-nmat-lcp}, and  Parthasarathy--Ravindran\cite{par-rav-nmat}. Coming to the other classes of matrices studied in this work: \textbf{(a)} The concept of almost $P$-matrices were introduced by Ky Fan in 1966 \cite{Fan66}. A characterization of almost $P$-matrices (with nonpsitive off diagonal entries) in terms of Linear Complementarity Problem was discussed by Miao \cite{M93}. \textbf{(b)} The notion of semipositive matrix was considered by Stiemke \cite{stiem} in connection with the problem of existence of positive solutions of linear systems. In 1994, Johnson, Kerr, and Stanford \cite{john-ker-stan-semi} introduced the notion of minimally semipositive matrices. These classes of matrices play a vital role in the study of $M$-matrices, in convergence theory for sets of matrices and in linear programming \cite{vander1972}. In this paper, we establish a characterization of almost $P$-matrices using the sign non-reversal property (Theorems \ref{alm-p-sec}, \ref{alm-p-1st}). We further obtain the interval hull characterization for all of these classes. Our results are summarized as follows:
\begin{center}
	
	\begin{tabular}{|c|c|l|l|c|}
		\hline
		
		{\bf Matrix Class} & {\bf Sign non-reversal} & { ~~~\bf $\mathbb{I}(A,B)$} & {\bf Testing set}~\\
		& ~~~~~~~\bf property&& {\bf for} $\mathbb{I}(A,B)$\\
		
			\hline $N$-matrices of the first&  ~\cite[Theorem 4.3]{moh-sri-nmat-lcp} & Theorem \ref{n hull 1st}&~$I_{z}, \ z \in \{ \pm 1 \}^n \setminus \{\pm e^J \}$\\
		
		category with respect to $J$ & && \\
		
		\hline $N$-matrices of the&  ~~\cite[Theorem 2]{par-rav-nmat} & Theorem \ref{n hull 2nd}&~$I_{z}, \ z \in \{ \pm 1 \}^n\setminus \{\pm e \}$\\
		
		second category & && \\
		
		\hline Almost $P$-matrices of the & &&\\
		first category with & ~~Theorem \ref{alm-p-1st} & Theorem \ref{almost p 1st hull} &~$I_{z}, \ z \in \{ \pm 1 \}^n $, $I_{P_J}$\\
		with respect to $J$ & && \\
		
		\hline Almost $P$-matrices & ~~Theorem \ref{alm-p-sec} & Theorem \ref{almst p sec hull}&~$I_{z}, \ z \in \{ \pm 1 \}^n$, $I_u$\\
		
		of the second category & && \\

		\hline Semipositive & ~~~~~~~N/A  & Theorem \ref{int-semi}&~$I_{l}$\\
		
		\hline Minimally semipositive  & ~~~~~~~N/A & Theorem \ref{int-semi}&~$I_{l}, I_{u}$\\
		
		\hline
		
	\end{tabular}
	\captionof{table}{Summary of results. Here, $J$ denotes a nonempty proper subset of $\langle n \rangle = \{1,\dots, n\}$.}\label{table1}
\end{center}

We conclude by explaining the notation used above.

\begin{defn}
Fix integers $m,n \geq 1$ and matrices $A, B \in \mathbb{R}^{m \times n}$, with interval hull $\mathbb{I}(A,B)$.
\begin{enumerate}
	\item Define the matrices
	\[
	(I_u)_{ij}:= \max\{a_{ij},b_{ij}\}, \qquad
	(I_l)_{ij}:= \min\{a_{ij},b_{ij}\}, \qquad
	I_c := \frac{B + A}{2}, \qquad
	\Delta := \frac{I_u - I_l}{2}.
	\]
	
	\item If $m=n$, given $z=(z_1,\dots,z_n)\in \{ \pm 1 \}^n$, define the matrices
	\[
	D_z :=\diag(z_1,\dots,z_n), \qquad
	I_z := I_c - D_z \Delta D_z.
	\]
\end{enumerate}
\end{defn}

This article is organized as follows: In section \ref{sec2}, we collect the needed definitions and known results.
 Section \ref{secnmat} contains  results for interval hull of $N$-matrices. In Section \ref{sec-alm-p-mat}, we establish a characterization of almost $P$-matrices via a sign non-reversal property and as a application, we study the interval hull of such matrices.  Section \ref{sec-semi} contains similar results for semipositive and minimally semipositive matrices.

\section{Preliminaries}\label{sec2}

We begin with notation, which will be used throughout the paper without
further reference.
 For a matrix $A \in \mathbb{R}^{m \times n}$, $A \geq 0 ~(A>0)$ signifies that all the components of the matrix $A$ are nonnegative (positive), and let $\vert A\vert := (\vert a_{ij}\vert)$.  For any positive integer $n$,
  define $\langle n \rangle := \{1,\dots, n\}$. Let $\mathbb{R}_\pm^n :=
  \{(x_1\dots,x_n) \in \mathbb{R}^{n}: \pm x_i \geq 0 ~\mbox{for all }~ i
  \in \langle n \rangle \}$.  For a subset $X$ of $\mathbb{R}^n$, let
  $\inte(X)$ denote the interior of $X$ in $\mathbb{R}^n$ with respect to
  the Euclidean metric. Let $e^i$ denote the vector whose $i$-th
  entry is $1$, and other entries are zero. 
   For $J\subseteq \langle n \rangle$, define $e^J \in\mathbb{R}^n$ such that $e^J_i=1$ for all $i \in J$ and $e^J_i=-1$ for all $i \notin J$. Also define $e^{\langle n\rangle}:= e$.

A square matrix $A$ is a \textit{$P$-matrix} if all its principal minors are positive. In \cite{gale-nikai-pmat}, Gale--Nikaid\^{o} established the sign non-reversal  property for $P$-matrices. \begin{theorem}[Sign non-reversal property]\label{snrp}
	A matrix $A \in \mathbb{R}^{n \times n}$ is a $P$-matrix if and only if $x \in \mathbb{R}^n$ and $x_i(Ax)_i \leq 0$ for all $i$ imply $x=0.$
\end{theorem}
Using the sign non-reversal property of $P$-matrices, in \cite{RJRG}, Rohn--Rex showed that the interval hull  of matrices $\mathbb{I}(A, B)$, where $A \leq B$, is a $P$-matrix, if a finite collection of matrices in  $\mathbb{I}(A, B)$ are $P$-matrices. Such a finite characterization of interval of $P$-matrices was first proved by Bia{\l}as and Garloff \cite{BGar84}, formulated in different terms.

In order to prove our results, we also require two basic lemmas. The first is a straightforward verification:

\begin{lemma}\label{int-lem1}
    Let $A,B \in \mathbb{R}^{m \times n}$. Then $I_l, I_u,
    \in \mathbb{I}(A,B)$. If $m=n$, then $I_z\in \mathbb{I}(A,B)$ for all $z\in \{ \pm 1 \}^n$.
\end{lemma}

The next lemma is precisely \cite[Theorem 2.1]{RJRG}
. We provide the proof for completeness.
\begin{lemma}\label{rohn_exten2}

    Let $A,B \in \mathbb{R}^{n \times n}$ and $x \in \mathbb{R}^n$. Let
    $z \in \{ \pm 1 \}^n$ such that $z_i = 1$ if $x_i \geq 0$ and $z_i = -1$ if $x_i < 0$. If $C \in \mathbb{I}(A,B)$, then
    \begin{center}
        $x_i(Cx)_i\geq x_i (I_z x)_i \hbox{~~ for  all~} i\in \langle n \rangle.$
    \end{center}

\end{lemma}

\begin{proof}
    Let  $C \in \mathbb{I}(A,B)$. Then $I_l\leq C \leq I _u$. Since $I_l= I_c- \Delta$ and $I_u= I_c + \Delta$, so \begin{center} $I_c -\Delta \leq C \leq I_{c} + \Delta $.\end{center} 
     Let $ x \in \mathbb{R}^n \setminus \{0\}$.  For fixed $1\leq i \leq n$, we have
    \begin{eqnarray}
    \vert   x_i((C-I_c)x)_i \vert & \leq &\vert x_i \vert (\vert C-I_c\vert \vert x \vert )_i
    \leq \vert x_i \vert ( \Delta  \vert x \vert )_i. \nonumber
    \end{eqnarray}
    Hence $$x_i (Cx)_i  \geq x_i (I_c x)_i - \vert x_i \vert ( \Delta  \vert x \vert )_i .$$
    Since $\vert x\vert=D_zx$, so $x_i (C x)_i \geq x_i((I_c - D_z \Delta D_z) x)_i$.
\end{proof}

\section{Results for $N$-matrices}\label{secnmat}

 We now characterize the interval hull property for $N$-matrices
(see Definition~\ref{maindefn}). First recall that an $N$-matrix $A$ is
of the \textit{first category} if it has at least one positive
entry. Otherwise, $A$ is of the \textit{second category}.

The following result gives a characterization for $N$-matrices of the
second category. This is known as the sign non-reversal property for the
$N$-matrices of the second category.

\begin{theorem}{\cite[Theorem 2]{par-rav-nmat}}\label{PR}
    Let $A \in \mathbb{R}^{n \times n}$. Then $A$ is an N-matrix of the second category if and only if $A<0$ and $A$ does not reverse the sign of any non-unisigned vector, that is, $x \in \mathbb{R}^n$ and  $x_i(Ax)_i\leq 0 $ for all $i$ imply $x \leq 0$ or $x \geq 0$.
\end{theorem}

 For $J \subseteq \langle n \rangle$, $J^\complement$ henceforth denotes $\langle n
 \rangle \setminus J$. For subsets $I, J \subset \langle n \rangle$ with
 elements arranged in ascending order, $A_{IJ}$ denotes the submatrix of
 $A$ whose rows and columns are indexed by $I$ and $J$, respectively.

\begin{theorem}{\cite[Theorem 4.3]{moh-sri-nmat-lcp}} \label{nmat-fir-char1}
    Let $A\in \mathbb{R}^{n \times n}$ be an $N$-matrix of the first category. Then $A$ can be written in the partitioned form (after a principal rearrangement of its rows and columns, if necessary)
    \begin{equation}\label{eq1}
    \begin{pmatrix}
        A_{JJ} & A_{J J^\complement}\\
        A_{J^\complement J} & A_{J^\complement J^\complement}
	\end{pmatrix},
    \end{equation} with $A_{JJ} < 0$, $A_{J^\complement J^\complement}<0$, $A_{JJ^\complement } >0$, and $A_{J^\complement J}>0$, where $J$  is a nonempty proper subset of $ \langle n \rangle$.

\end{theorem}
\begin{defn}\label{defnn1stj}
Let $J$  be a nonempty proper subset of $ \langle n \rangle$.
An $N$-matrix $A$ is  \emph{an $N$-matrix of the first category
with respect to $J$} if  it is of the form  $\begin{pmatrix}
	A_{JJ} & A_{J J^\complement}\\
	A_{J^\complement J} & A_{J^\complement J^\complement} \\
	\end{pmatrix}$ with $A_{JJ} < 0$, $A_{J^\complement J^\complement}<0$, $A_{JJ^\complement } >0$, and $A_{J^\complement J}>0$.
\end{defn}
The next result gives a characterization for $N$-matrices of the first category with respect to $J$, and is known as the sign non-reversal property for such matrices.

\begin{theorem}{\cite[Theorem 4.3]{moh-sri-nmat-lcp}} \label{nmat-fir-char}
    Let $J$ be a nonempty proper subset of $\langle n \rangle$ and let $A = \begin{pmatrix}
        A_{JJ} & A_{J J^\complement}\\
        A_{J^\complement J} & A_{J^\complement J^\complement} \\
        \end{pmatrix}\in \mathbb{R}^{n \times n},$ where $A_{JJ} < 0$, $A_{J^\complement J^\complement}<0$, $A_{J J^\complement } >0$, and $A_{J^\complement J}>0.$ Then $A$  is an $N$-matrix of the first category with respect to $J$ if and only if  $A$ reverses the sign of a vector $x \in \mathbb{R}^n$, i.e, $x_i(Ax)_i\leq 0$ for all $i \in \langle n \rangle$, then either $x_J \leq 0$ and $x_{J^\complement} \geq 0$, or $x_J \geq 0$ and $x_{J^\complement} \leq 0.$

\end{theorem}
First, we  present a necessary condition for interval hull of $N$-matrices similar in spirit to that of Theorem \ref{PR} and \ref{nmat-fir-char1}. Notice that the inequality holds uniformly here.
\begin{theorem}\label{unieqn2}
	Let $A,B \in \mathbb{R}^{n \times n}$ such that  $\mathbb{I}(A,B)$ is an $N$-matrix of the second category. Then for each $x \in \mathbb{R}^n \setminus \{0\}$ with $x \ngeq 0$ and $ x\nleq 0$, there exists $i \in \langle n \rangle$ such that
	\begin{center}
		$x_i(Cx)_i>0$ for all $C\in \mathbb{I}(A,B).$
	\end{center}
\end{theorem}
\begin{proof}
	Let $x \in \mathbb{R}^n \setminus \{0\}$ with $x \ngeq 0$ and $x \nleq 0$.  Let
	$z \in \{ \pm 1 \}^n$ such that $z_i = 1$ if $x_i \geq 0$ and $z_i = -1$ if $x_i < 0$. By Lemma \ref{int-lem1}, $I_z$ is $N$-matrix of the second category. By Theorem \ref{PR}, there exists $i \in \langle n \rangle$ such that $x_i(I_z x)_i>0$. Thus by Lemma \ref{rohn_exten2}, $x_i(Cx)_i \geq x_i(I_z x)_i>0$ for all $C\in \mathbb{I}(A,B)$.
\end{proof}

\begin{rem}\label{uniremn1} The above result has an analogue for $N$-matrices of the first category with respect to $J$ for each $x \in \mathbb{R}^n \setminus \{0\}$ with $x_J \nleq 0$ or $x_{J^\complement} \ngeq 0$, and $x_J \ngeq 0$ or $x_{J^\complement} \nleq 0$. We leave the details to the interested reader.
	\end{rem}

We now characterize interval hulls of $N$-matrices by reducing it to a finite set of test matrices, beginning with those of the second category.
\begin{theorem}\label{n hull 2nd}
	Let $A,B \in \mathbb{R}^{n \times n}$ such that $\max\{a_{ii}, b_{ii}\}<0$ for all $i \in \langle n \rangle$.
	Then, $\mathbb{I}(A,B)$ is an $N$-matrix of the second category  if and only if $I_z$ is an $N$-matrix of the second category for all $z\in \{ \pm 1 \}^n\setminus \{\pm e\}$.
\end{theorem}

\begin{proof}
	Suppose $\mathbb{I}(A,B)$ is an $N$-matrix of the second category. By Lemma \ref{int-lem1}, $I_z$ is an $N$-matrix of the second category.
	 
	 Conversely, suppose that $I_z$ is an $N$-matrix of the second category for all $z\in \{ \pm 1 \}^n\setminus \{\pm e\}$. Let $C \in \mathbb{I}(A,B)$. First we show that $C<0$. For $i \in \langle n \rangle$, define $z^i\in \{ \pm 1 \}^n$ such that $z^i_i=1$ and $z^i_j=-1$ for $j \neq i$. Then $I_{z^i}<0$ for all $i \in \langle n \rangle$, since $I_{z^i}$ is an $N$-matrix of the second category. Thus $(I_u)_{ij}<0$ for $j \in \langle n \rangle$. Hence $C\leq I_u<0$. Let $x \in \mathbb{R}^n$ such that $x \ngeq 0$ and $x \nleq 0$. By Theorem \ref{unieqn2}, there exists  $i \in \langle n \rangle$ such that $x_i (C x)_i>0$. Hence $C$ is an $N$-matrix of the second category by Theorem \ref{PR}.
\end{proof}

 We next characterize the interval hull of $N$-matrices of the first category with respect to $J$, where $J\subseteq \langle n \rangle$.
\begin{theorem}\label{n hull 1st}
	Let $J$ be a nonempty proper subset of $\langle n \rangle$, and let $A,B \in \mathbb{R}^{n \times n}$ such that $\max\{a_{ii}, b_{ii}\}<0$ for all $i \in \langle n \rangle$. Then, $\mathbb{I}(A,B)$ is an $N$-matrix of the first category with respect to $J$  if and only if $I_z$ is an $N$-matrix of the first category with respect to $J$ for all $z\in \{ \pm 1 \}^n\setminus \{\pm e^J\}$.
\end{theorem}
\begin{proof}
	The `if' part is immediate from Lemma \ref{int-lem1}.
	Conversely, suppose that $I_z$ is an $N$-matrix of the first category with respect to $J$ for all $z\in \{ \pm 1 \}^n\setminus \{\pm e^J\}$. Let $C \in \mathbb{I}(A,B)$. First we show that $C$ can be partitioned as $C = \begin{pmatrix}
		C_{JJ} & C_{J J^\complement}\\
		C_{J^\complement J} & C_{J^\complement J^\complement} \\
		\end{pmatrix},$ where $C_{JJ}, C_{J^\complement J^\complement}< 0$, and  $C_{J J^\complement }, C_{J^\complement J }>0$. For $i \in \langle n \rangle$, define $z^i\in \{ \pm 1 \}^n$ as follows: $z^i_i=1$ and if $i \in J$ and $j \neq i$ then 
		\[
		z^i_j=\left\{
		\begin{array}{cc}
		-1 &\hspace{5mm} j \in J,\\
		1 &\hspace{5mm} \mbox{otherwise}.
		\end{array}
		\right.
		\]
		If $i \notin J$ and $j \neq i$ then \[
		z^i_j=\left\{
		\begin{array}{cc}
		1 &\hspace{5mm} j \in J,\\
		-1 &\hspace{5mm} \mbox{otherwise}.
		\end{array}
		\right.
		\]
 Since $I_{z^i}$ is an $N$-matrix of the first category with respect to $J$ for all $i \in \langle n \rangle$, so  $I_{z^i}$ can be partitioned as $I_{z^i} = \begin{pmatrix}
		{I_{z^i}}_{JJ} & {I_{z^i}}_{J J^\complement}\\
		{I_{z^i}}_{J^\complement J} & {I_{z^i}}_{J^\complement J^\complement} \\
		\end{pmatrix},$ where ${I_{z^i}}_{JJ}, {I_{z^i}}_{J^\complement J^\complement}< 0$, and  ${I_{z^i}}_{J J^\complement }, {I_{z^i}}_{J^\complement J }>0$. Thus ${I_{u}}_{JJ}, {I_{u}}_{J^\complement J^\complement}< 0$, and  ${I_{l}}_{J J^\complement }, {I_{l}}_{J^\complement J }>0$. Hence  $C_{JJ}, C_{J^\complement J^\complement}< 0$, and  $C_{J J^\complement }, C_{J^\complement J }>0$. Let $x \in \mathbb{R}^n$ such that $x_J \nleq 0$ or $x_{J^\complement} \ngeq 0$, and $x_J \ngeq 0$ or $x_{J^\complement} \nleq 0.$ By Lemma \ref{rohn_exten2}, there exists  $z\in \{ \pm 1 \}^n\setminus \{\pm e^J\}$ such that $x_i (C x)_i \geq x_i(I_z x)_i$, for $i\in \langle n \rangle$. Since $I_z$ is an $N$-matrix of the first category with respect to $J$,  by Theorem \ref{nmat-fir-char},  there exists  $i \in \langle n \rangle$ such that $0<x_i(I_z x)_i<x_i (C x)_i$. Thus $C$ is an $N$-matrix of the first category with respect to $J$ by Theorem \ref{nmat-fir-char}.
\end{proof}

\section{Results for almost $P$-matrices}\label{sec-alm-p-mat}

We now establish the sign non-reversal property for almost $P$-matrices
(see Definition \ref{maindefn}) and characterize their interval hull.
Recall that an $n \times n ~(n\geq 2)$ matrix $A$ is an almost $P$-matrix
if and only if $A^{-1}$ is an $N$-matrix \cite[Lemma 2.4]{koji-sai-1979}.
Motivated by this result, we classify almost $P$-matrices into two categories:
\begin{defn}\label{defnalmost p}
	Let $n\geq 2$ be an integer.
	\begin{itemize}
		\item[(i)] Let $J$  be a nonempty proper subset of $ \langle n \rangle$. An almost $P$-matrix $A$ is \emph{an almost $P$-matrix of the first category with respect to $J$} if $A^{-1}$ is an $N$-matrix of the first category with respect to $J$.
		\item[(ii)] An almost $P$-matrix $A$ is an \emph{almost $P$-matrix of the second category} if $A^{-1}$ is an $N$-matrix of the second category.
	\end{itemize}
\end{defn}
Observe that if $A$ is an almost $P$-matrix of the second category, then there exists a
positive vector $x$ such that $Ax<0$. Our next result shows a sign
non-reversal property for such matrices.

\begin{theorem}\label{alm-p-sec}
    Let $A \in \mathbb{R}^{n \times n}$. Then, $A$ is an almost $P$-matrix of the second category if and only if the following hold:

    \begin{enumerate}
     \item[(a)] $N(A) \cap  \inte (\mathbb{R}_+^n) = \emptyset$, where
 $N(A)$  denotes the null space of $A$,
     \item[(b)] $x_i(Ax)_i \leq 0$ for all $i \in \langle n \rangle$ imply that $x =0$, if  $x_k =0$ for some $k$; otherwise $x > 0 $ or $x <0$,
     \item[(c)] $A (\inte (\mathbb{R}_+^n)) \cap \inte (\mathbb{R}_-^n) \neq \emptyset$.
    \end{enumerate}
\end{theorem}

\begin{proof}
    Let $A$ be an almost $P$-matrix of the second category. Then $(a)$ holds trivially. Let $x \in \mathbb{R}^n$ such that $x_i (Ax)_i \leq 0$ for all $i$. Suppose that  $x_k =0$ for some $k$. Let $B$ be the principal submatrix of $A$ obtained by deleting the $k$-th row and the $k$-th column of $A$. Let $y$ be the $(n-1)$-vector obtained from the vector $x$ by deleting the $k$-th entry. Then  $y_j (By)_j \leq 0$ for all $j$. Since $B$ is a $P$-matrix, so $y_j=0$ for all $j$.  Thus $x =0.$ Suppose that $x_i \neq 0$ for all $i$.  Let $y = Ax$. Then, $y_i(A^{-1}y)_i \leq 0$  for all $i$. Since $A^{-1}$ is an $N$-matrix of the second category, so, by Theorem \ref{PR}, either $y \geq 0$ or $y \leq 0$.  Note that $A^{-1}<0$. Hence, all the components of the vector $x = A^{-1} y$ are   either positive or negative. As $A$ is an almost $P$-matrix of the second category, there exists a positive vector $x$ such that $Ax<0$. Thus $(c)$ holds.

To prove the converse, first let us show that all the proper principal minors are positive.  Let $B$ be any $(n-1) \times (n-1) $ principal submatrix of $A$. Without loss of generality, assume that $B$ is obtained from  $A$ by deleting the last row and last column of $A$. Let $y \in \mathbb{R}^{n-1}$, and  define $x =  \begin{pmatrix}
            y\\
            0 \\
            \end{pmatrix}  $.
If $y_i(By)_i \leq 0$ for all $i$, then $x_i (Ax)_i \leq 0$ for all $i$.
Thus, by the assumption,  $x = 0$, and hence $y =0$. By Theorem
\ref{snrp}, all the proper principal minors of $A$ are positive. We now
claim that $A$ is invertible. Let $x$ be a vector such that $Ax=0$. Then either $x>0$ or $x<0$ or $x=0$. Since $N(A) \cap  \inte (\mathbb{R}_+^n) = \emptyset$, so $x =0$. Also $A$ is not a $P$-matrix, since $A (\inte (\mathbb{R}_+^n)) \cap \inte (\mathbb{R}_-^n) \neq \emptyset$, so $A$ reverses the sign of a nonzero vector. Thus $\det A < 0$, hence $A$ is an almost $P$-matrix. Let $x  = A^{-1} e_i$.
 Note that the $i$-th entry of the vector $x$ is negative. Now, we have  $x_j (Ax)_j \leq 0$ for all $j \in \langle n \rangle$. Thus the vector $x$ is entrywise negative, and hence $A^{-1} < 0$. So $A$ is an almost $P$-matrix of the second category.
 \end{proof}

 Using Theorem \ref{alm-p-sec}, we establish an equivalent condition for an  interval hull to be a subset of the set of all almost $P$-matrices of the second category.

\begin{theorem}\label{almst p sec hull}
    Let $A,B \in \mathbb{R}^{n \times n}$. Then $\mathbb{I}(A,B)$ is an almost $P$-matrix of the second category  if and only if $I_u$ and $I_z$ are almost $P$-matrices of the second category for all $z\in \{ \pm 1 \}^n$.
\end{theorem}

\begin{proof}
    If $\mathbb{I}(A,B)$ is an almost $P$-matrix of the second category, then so are $I_u$ and $I_z$ by Lemma \ref{int-lem1}.

    Conversely, suppose that $I_u$ and $I_z$ are almost $P$-matrices of the second category for all $z\in \{ \pm 1 \}^n$. Let $C \in \mathbb{I}(A,B)$. First let us show that $N(C) \cap  \inte (\mathbb{R}_+^n)= \emptyset$. Let $x \in N(C) \cap  \inte (\mathbb{R}_+^n)$. Then $I_ux\geq 0$, since $I_u\geq C$ and $Cx=0$. Since $I^{-1}_u<0$, so $I^{-1}_u(I_ux)=x\leq 0$, a contradiction. Thus $N(C) \cap  \inte (\mathbb{R}_+^n)= \emptyset$.

     Since $I_u$ is an almost $P$-matrix of the second category, there exists a vector $x>0$ such that $I_ux<0$. Thus $Cx<0$, since $C \leq I_u$. Hence $C(\inte (\mathbb{R}_+^n)) \cap \inte (\mathbb{R}_-^n) \neq \emptyset$.

 Let $x \in \mathbb{R}^n$ be such that $x_i(Cx)_i\leq 0$ for $i \in \langle   n \rangle$. By Lemma \ref{rohn_exten2}, there exists a vector $z\in \{ \pm 1 \}^n$ such that $0\geq x_i (C x)_i \geq x_i(I_z x)_i$ for $i \in \langle   n \rangle$. Since $I_z$ is an almost $P$-matrix of the second category, by Theorem \ref{alm-p-sec}, $x=0$, if $x_k = 0$ for some $k$,  otherwise $x< 0$ or $x>0$. Thus, by Theorem \ref{alm-p-sec}, $C$ is an almost $P$-matrix of the second category.
\end{proof}

For a nonempty proper subset $ J$ of $ \langle n \rangle$, define
\begin{equation}
\mathbb{J}  = \mathbb{J}_J := \{x \in \mathbb{R}^n: x_j > 0
~\mbox{for all}~ j \in J ~\mbox{and}~ x_j <0 ~\mbox{for all}~ j \notin J
\}.
\end{equation}
Note that if $A$ is an almost $P$-matrix of the first category with respect to $J$, then there exists  $x \in \mathbb{J}$ such that $Ax \in -\mathbb{J}$. Next, we develop the sign non-reversal property for almost $P$-matrices of the first category with respect to $J$.

 \begin{theorem}\label{alm-p-1st}
 Let $A \in \mathbb{R}^{n \times n}$ and let $J$  be a nonempty proper subset of $ \langle n \rangle$. Then $A$ is an almost $P$-matrix of the first category with respect to $J$ if and only if
the following hold:

 \begin{enumerate}
 	\item[(a)] $N(A) \cap  \mathbb{J} = \emptyset$,
 	\item[(b)] $x \in \mathbb{R}^n$ and $x_i (Ax)_i \leq 0$ for all $i \in \langle n \rangle$ imply $x=0,$ if $x_k=0$  for some $k$; otherwise either $x\in \mathbb{J}$ or $-x \in \mathbb{J}$,
 	\item[(c)] $A (\mathbb{J}) \cap  (-\mathbb{J}) \neq \emptyset$.
 \end{enumerate}
 \end{theorem}
 \begin{proof}
Let $A$ be an almost $P$-matrix of the first category with respect to $J$. Then,
$A^{-1} = \begin{pmatrix}
            B_{JJ} & B_{J J^\complement} \\
B_{J^\complement J} & B_{J^\complement J^\complement}\\
            \end{pmatrix}$, where $B_{JJ} < 0, ~B_{J^\complement J^\complement} < 0, ~B_{J J^\complement} >0$ and $B_{J^\complement J} > 0$. Let  $x_i(Ax)_i \leq 0.$ If $x_i =0$ for some $i$, then, by an  argument similar to that of Theorem \ref{alm-p-sec}, we get $x=0$.  So, let us assume that $x_i \neq 0$ and $x_i(Ax)_i \leq 0$ for each $i$. Let $y=Ax$. Then $y_i(A^{-1}y)_i \leq 0$ for each $i$.    Hence, by Theorem \ref{nmat-fir-char}, $y=
\begin{pmatrix}
            y_{J}\\
            y_{J^\complement}\\
            \end{pmatrix} $ where either $y_J \leq 0$, $y_{J^\complement} \geq 0$ or $y_J \geq 0$, $y_{J^\complement} \leq 0$. Also, $x=A^{-1}y = \begin{pmatrix}
            x_J\\
            x_{J^\complement}\\
            \end{pmatrix} $, so $x_J >0, x_{J^\complement} <0$ or $x_J <0, x_{J^\complement} >0.$ As $A$ is an almost $P$-matrix of the first category with respect to $J$, so there exists  $x \in \mathbb{J}$ such that $Ax \in -\mathbb{J}$. Thus $(c)$ holds.

Conversely, suppose that $A$ satisfies all the three properties as above. Let $C$ be any $(n-1) \times (n-1) $ principal submatrix of $A$. Without loss of generality, assume $C$ is obtained from  $A$ by deleting the last row and last column of $A$.
Let $y \in \mathbb{R}^{n-1}.$ Set $x= \begin{pmatrix} y\\ 0\\
\end{pmatrix}.$  If $y_i(Cy)_i \leq 0$ for all $i$, then $x_i
(Ax)_i \leq 0$ for all $i \in \langle n \rangle$. Thus, by assumption $x
= 0$, so $y =0$. By Theorem \ref{snrp}, all the proper principal minors of $A$ are
positive. Since $N(A) \cap \mathbb{J} = \emptyset$, so $A$ is invertible
by a similar argument to that of Theorem \ref{alm-p-sec}. As,  $A (\mathbb{J}) \cap
(-\mathbb{J}) \neq \emptyset$, so the matrix $A$ is not  a $P$-matrix by Theorem \ref{snrp}.
Thus  $A$ is an almost $P$-matrix. Let $x=A^{-1}e_k$. Then $x_j
(Ax)_j\leq 0$ for $j\in \langle n \rangle$. If $k\in J$, then
$(A^{-1}e_k)_J<0$ and $(A^{-1}e_k)_{J^\complement}>0$, since the $k$-th
entry of $x$ is negative. Otherwise $(A^{-1}e_k)_J>0$ and
$(A^{-1}e_k)_{J^\complement}<0$. Hence $A$ is an almost $P$-matrix of the
first category with respect to $J$.
\end{proof}

For a nonempty proper subset $J$ of  $\langle n \rangle$, define the
matrices
\begin{equation}
I_J := {\rm diag}((-1)^{1 (i \not\in J)} : i \in \langle n \rangle), \qquad
I_{P_J}:=I_c + I_J \Delta I_J.
\end{equation}
Also define $I_\emptyset := -{\rm Id}_n$. One can verify that  $I_{P_J}
\in \mathbb{I}(A,B)$.

In the following theorem, we characterize when an interval hull is a subset of the set of all almost $P$-matrices of the first category with respect to $J$.  
\begin{theorem}\label{almost p 1st hull}
 Let $A,B \in \mathbb{R}^{n \times n}$. Then $\mathbb{I}(A,B)$ is an almost $P$-matrix of the first category with respect to $J$ if and only if $I_{P_J}$ and $I_z$ are almost $P$-matrices  of the first category with respect to $J$ for all $z\in \{ \pm 1 \}^n$.
\end{theorem}

\begin{proof}
    Let $\mathbb{I}(A,B)$ be an almost $P$-matrix of the first category with respect to $J$.   Then the matrices $I_{P_J}$ and $I_z$ are almost $P$-matrices of the first category with respect to $J$

    Conversely, suppose that the matrices $I_{P_J}$ and $I_z$ are almost $P$-matrices of the first category with respect to $J$  for all $z\in \{ \pm 1 \}^n$. From the definition, $I_{P_J}=\begin{pmatrix}
            {(I_u)}_{JJ} & {(I_l)}_{J J^\complement} \\
            {(I_l)}_{J^\complement J} & {(I_u)}_{J^\complement  J^\complement}\\
    \end{pmatrix}$. Let $C \in \mathbb{I}(A,B)$. First let us show that $N(C) \cap  \mathbb{J}= \emptyset$. Let $x \in N(C) \cap  \mathbb{J}$ and $y=I_{P_J}x$. Then  $x_J>0$ and $x_{J^\complement}<0$. Now,  $(I_u)_{JJ} x_{J} + (I_l)_{J J^\complement}x_{J^\complement} \geq  C_{JJ} x_{J} + C_{J J^\complement}x_{J^\complement}$, $(I_l)_{J^\complement J} x_{J} + (I_u)_{J^\complement J^\complement}x_{J^\complement} \leq  C_{J^\complement J} x_{J} + C_{J^\complement J^\complement}x_{J^\complement}$ and $Cx=0$ imply that $y_J\geq0$ and $y_{J^\complement}\leq 0$. Since ${(I^{-1}_{P_J})}_{JJ}<0, {(I^{-1}_{P_J})}_{J^\complement J^\complement} < 0$,  ${(I^{-1}_{P_J})}_{JJ^\complement }>0$ and  $ {(I^{-1}_{P_J})}_{J^\complement J }>0$, so ${(I^{-1}_{P_J}y)}_J=x_J\leq  0$ and ${(I^{-1}_{P_J}y)}_{J^\complement}=x_{J^\complement} \geq 0$, a contradiction. Thus $N(C) \cap  \mathbb{J}= \emptyset$.

    Since $I_{P_J}$ is an almost $P$-matrix of the first category with respect to $J$, there exists a vector $x\in \mathbb{J}$ such that ${I_{P_J}x} \in -\mathbb{J}$. 
     Thus ${(Cx)}_J<0$ and ${(Cx)}_{J^\complement}>0$.
    Hence $C (\mathbb{J}) \cap  (-\mathbb{J}) \neq \emptyset$.

      Let $x \in \mathbb{R}^n$ be  such that $x_i(Cx)_i\leq 0$ for $i \in \langle n \rangle$. By Lemma \ref{rohn_exten2}, there exists a vector $z\in \{ \pm 1 \}^n$ such that $x_i(I_z x)_i\leq x_i (C x)_i\leq 0$, for $i\in \langle n \rangle$. Since $I_z$ is an almost $P$-matrix of the first category with respect to $J$, by Theorem \ref{alm-p-1st}, $x=0$, if $x_k = 0$ for some $k$,  otherwise $x\in \mathbb{J}$ or $x \in -\mathbb{J}$. Hence, by Theorem \ref{alm-p-1st}, $C$ is an almost $P$-matrix of the first category with respect to $J$.
\end{proof}
\begin{rem}
	Theorem \ref{unieqn2} has analogue for almost $P$-matrices of either category with appropriate choices of $x$. We leave the details to the interested reader.
\end{rem}

    \section{Rsesults for semipositive matrices}\label{sec-semi}
In this section we characterize the interval hull for semipositive and minimally semipositive matrices (see Definition \ref{maindefn}). In \cite{CKS18}, the authors studied the interval hull $\mathbb{I}(A,B)$ of minimally semipopsitive matrices, where $A\leq B$. In this article, we provide a short and elementary proof.
\begin{theorem}\label{int-semi}
    Let $A,B \in \mathbb{R}^{m \times n}$. Then we have the following:
 \begin{itemize}
        \item[(a)] $\mathbb{I}(A,B)$ is semipositive  if and only if $I_l$ is a semipositive matrix.
        \item[(b)] $\mathbb{I}(A,B)$ is minimally semipositive if and only if $I_l$ is semipositive and $I_u$ is minimally semipositive.
\end{itemize}

\end{theorem}

\begin{proof}
    \textbf{(a)} If $\mathbb{I}(A,B)$ is semipositive, then $I_l$ is also semipositive. Conversely, suppose that $I_l$ is semipositive. Then, there exists $x\geq 0$ such that $I_lx>0$. Let $C \in \mathbb{I}(A,B)$. Then $0< I_l x\leq Cx$. Thus $\mathbb{I}(A,B)$ is semipositive.

    \textbf{(b)} Suppose that $I_l$ is semipositive and $I_u$ is minimally semipositive. Then $\mathbb{I}(A,B)$ is semipositive. Suppose
    $C \in\mathbb{I}(A,B)$ is not minimally semipositive. Then there exists a  nonnegative nonzero vector $x$ with at least one zero entry such that $Cx>0$. Since $I_u\geq C$, so $I_u x>0$, a contradiction. Thus $\mathbb{I}(A,B)$ is minimally semipositive.
\end{proof}

It is known that a square matrix $A$ is minimally semipositive if and
only if $A$ is invertible and $A^{-1}\geq 0$. More generally,  an $m \times n$ matrix $A$ is minimally semipositive  if and only if $A$ is semipositive and $A$ has a nonnegative left inverse \cite{john-ker-stan-semi}. This leads to the following result:

\begin{theorem}[{\cite[Theorem 25.4]{kras-lif-sob}}]\label{kras}
Let $B,C \in \mathbb{R} ^{n \times n}$ such that $C \leq B$, $B$ is
invertible, and $B^{-1}\geq 0$.
Then $C^{-1} \geq 0$ if and only if $\inte(\mathbb{R}_+^n) \cap C\mathbb{R}_+^n \neq \emptyset$.
\end{theorem}

Indeed, in \cite{kras-lif-sob}, the authors proved the above theorem for
any normal and solid cone in $\mathbb{R}^n$. Now, it is clear that part
(b) of Theorem \ref{int-semi} is an extension of Theorem \ref {kras} for
rectangular matrices. We also observe that our argument provides an
alternate, simpler, and elementary proof for Theorem \ref{kras}.

\section*{Acknowledgements}
P.N. Choudhury was supported by National Post-Doctoral
Fellowship(PDF/2019/000275), the SERB, Department of Science and
Technology, India, and the NBHM Post-Doctoral Fellowship
(0204/11/2018/R$\&$D-II/6437) from DAE (Govt. of India). M.R. Kannan
would like to thank  the SERB, Department of Science and Technology,
India, for financial support through the projects  MATRICS
(MTR/2018/000986) and Early Career Research Award (ECR/2017/000643).

\bibliographystyle{plain}
\bibliography{pro-raj}

\begin{thebibliography}{10}

\bibitem{BGar84}
S.~Bia{\l}as and J.~Garloff.
\newblock Intervals of {$P$}-matrices and related matrices.
\newblock {\em Linear Algebra Appl.}, 58:33--41, 1984.

\bibitem{CKS18}
P.N. Choudhury, M.R. Kannan, and K.C. Sivakumar.
\newblock New contributions to semipositive and minimally semipositive
  matrices.
\newblock {\em Electron. J. Linear Algebra}, 34:35--53, 2018.

\bibitem{PT20}
P.N. Choudhury and M.J. Tsatsomeros.
\newblock Algorithmic detection and construction of {N}-matrices.
\newblock {\em Linear Algebra Appl.}, 602:46--56, 2020.

\bibitem{Fan66}
Ky~Fan.
\newblock Some matrix inequalities.
\newblock {\em Abh. Math. Sem. Univ. Hamburg}, 29:185--196, 1966.

\bibitem{gale-nikai-pmat}
D.~Gale and H.~Nikaid\^{o}.
\newblock The {J}acobian matrix and global univalence of mappings.
\newblock {\em Math. Ann.}, 159:81--93, 1965.

\bibitem{GAT16}
J.~Garloff, M.~Adm, and J.~Titi.
\newblock A survey of classes of matrices possessing the interval property and
  related properties.
\newblock {\em Reliab. Comput.}, 22:1--14, 2016.

\bibitem{I71}
K-i Inada.
\newblock The production coefficient matrix and the {S}tolper-{S}amuelson
  condition.
\newblock {\em Econometrica}, 39:219--239, 1971.

\bibitem{john-ker-stan-semi}
C.R. Johnson, M.K. Kerr, and D.P. Stanford.
\newblock Semipositivity of matrices.
\newblock {\em Linear and Multilinear Algebra}, 37:265--271, 1994.

\bibitem{koji-sai-1979}
M.~Kojima and R.~Saigal.
\newblock On the number of solutions to a class of linear complementarity
  problems.
\newblock {\em Math. Programming}, 17:136--139, 1979.

\bibitem{kras-lif-sob}
M.A. Krasnoselskij, Je.A. Lifshits, and A.V. Sobolev.
\newblock {\em Positive linear systems}, volume~5 of {\em Sigma Series in
  Applied Mathematics}.
\newblock Heldermann Verlag, Berlin, 1989.
\newblock The method of positive operators, from the Russian by J\"{u}rgen
  Appell.

\bibitem{Ku71}
J.R. Kuttler.
\newblock A fourth-order finite-difference approximation for the fixed membrane
  eigenproblem.
\newblock {\em Math. Comp.}, 25:237--256, 1971.

\bibitem{M93}
J.M. Miao.
\newblock Ky {F}an's {$N$}-matrices and linear complementarity problems.
\newblock {\em Math. Programming}, 61:351--356, 1993.

\bibitem{moh-sri-nmat-lcp}
S.R. Mohan and R.~Sridhar.
\newblock On characterizing {$N$}-matrices using linear complementarity.
\newblock {\em Linear Algebra Appl.}, 160:231--245, 1992.

\bibitem{par-rav-nmat}
T.~Parthasarathy and G.~Ravindran.
\newblock {$N$}-matrices.
\newblock {\em Linear Algebra Appl.}, 139:89--102, 1990.

\bibitem{rohn1}
J.~Rohn.
\newblock Positive definiteness and stability of interval matrices.
\newblock {\em SIAM J. Matrix Anal. Appl.}, 15:175--184, 1994.

\bibitem{RJRG}
J.~Rohn and G.~Rex.
\newblock Interval {$P$}-matrices.
\newblock {\em SIAM J. Matrix Anal. Appl.}, 17:1020--1024, 1996.

\bibitem{stiem}
E.~Stiemke.
\newblock \"{U}ber positive {L}\"osungen homogener linearer {G}leichungen.
\newblock {\em Math. Ann.}, 76:340--342, 1915.

\bibitem{vander1972}
J.S. Vandergraft.
\newblock Applications of partial orderings to the study of positive
  definiteness, monotonicity, and convergence of iterative methods for linear
  systems.
\newblock {\em SIAM J. Numer. Anal.}, 9:97--104, 1972.

\end{thebibliography}
\end{document}